\documentclass[12pt]{article}
\usepackage{amsmath,amsfonts,amssymb}
\usepackage{alltt}

\usepackage{amsmath,amsfonts,ifthen,fullpage,enumerate}  
\usepackage{mathrsfs}
\usepackage{mathtools}

\usepackage{enumerate}

\usepackage{amsmath,amsfonts,amssymb}

\usepackage{makecell}  

\usepackage{graphicx}        

\usepackage[table]{xcolor}   

\def\spam{\mathop{\rm span}\nolimits}

\def\Re{\mathop{\rm Re}\nolimits}

\def\Harm{\mathop{\rm Harm}\nolimits}

\def\ran{\mathop{\rm ran}\nolimits}

\def\question#1{{\bf Question: }#1}
\def\question#1{}

\def\R{\mathbb{R}}

\def\CC{\mathbb{C}}

\def\HH{\mathbb{H}}

\def\Cd{\C^d}

\def\Hd{\HH^d}

\def\Rd{\R^d}

\def\C{\mathbb{C}}

\newcommand{\RR}{\mathbb{R}}

\newtheorem{theorem}{Theorem}[section]
\newtheorem{corollary}{Corollary}[section]
\newtheorem{lemma}{Lemma}[section]
\newtheorem{example}{Example}[section]

\newenvironment{proof}{{\noindent \it
Proof.}}{\hfill$\Box$\medskip}
\newif\ifdraft\def\draft{\drafttrue\hoffset=.8truecm\showlabeltrue
\def\comment##1{{\bf comment: ##1}}
\headline={\sevenrm \hfill \ifx\filenamed\undefined\jobname\else\filenamed\fi%
(.tex) (as of \ifx\updated\undefined???\else\updated\fi)
 \TeX'ed at {\hour\time\divide\hour by 60{}%
\minutes\hour\multiply\minutes by 60{}%
\advance\time by -\minutes
\the\hour:\ifnum\time<10{}0\fi\the\time\  on \today\hfill}}
}

\def\ip#1{\langle\kern-.28em\langle#1\rangle\kern-.28em\rangle_\nu}

\def\openR{{{\rm I}\kern-.16em {\rm R}}}

\let\ga\alpha

\let\gb\beta

\let\gd\delta
\let\gD\Delta

\let\gTh\Theta

\let\gs\sigma

\let\ga\alpha

\let\gb\beta
\let\gd\delta
\let\gs\sigma

\def\Hom{\mathop{\rm Hom}\nolimits}

\def\Pol{\mathop{\rm Pol}\nolimits}
\def\Harm{\mathop{\rm Harm}\nolimits}
\def\Reg{\mathop{\rm Reg}\nolimits}

\def\Iff{\hskip1em\Longleftrightarrow\hskip1em}
\def\Implies{\hskip1em\Longrightarrow\hskip1em}

\def\formeq{\the\sectionno.\the\equationno}  
\def\elabel#1/#2/#3/{\global\advance\equationno by 1 %
\ifx#1\empty\else\emember#1%
\ifshowlabel\marginal{\string#1}\fi\fi%
\ifmmode\eqno{#3(\formeq#2)}\else#3\formeq#2\fi} 

\def\makeblanksquare#1#2{
\dimen0=#1pt\advance\dimen0 by -#2pt
      \vrule height#1pt width#2pt depth0pt\kern-#2pt
      \vrule height#1pt width#1pt depth-\dimen0 \kern-#1pt
      \vrule height#2pt width#1pt depth0pt \kern-#2pt
      \vrule height#1pt width#2pt depth0pt
}

\title{\bf 
Multivariate Lagrange interpolation and polynomials of
one quaternionic variable
}

\author{Shayne Waldron\\ }

\begin{document}

\maketitle 

\begin{abstract}

This paper considers the extension of classical Lagrange interpolation
in one real or complex variable to ``polynomials of one quaternionic
variable''. To do this we develop some aspects of the theory of such polynomials. 
We then give a number of related multivariate polynomial 
interpolation schemes for $\RR^4$ and $\CC^2$ with good geometric properties, 
and some aspects of least interpolation and of Kergin interpolation.

\end{abstract}

\bigskip
\vfill


 \ifdraft Skew fields, division rings \else\fi
41A05, \ifdraft Interpolation in approximation theory \else\fi   
41A10, \ifdraft Approximation by polynomials \else\fi   
41A63, \ifdraft Multidimensional problems \else\fi   


secondary
12E10, \ifdraft Special polynomials in general fields \else\fi
16D10. \ifdraft General module theory in associative algebras \else\fi

\vskip .5 truecm
\hrule
\newpage

\section{Introduction}

The {\bf quaternions} $\HH$ are a celebrated extension of the field of complex numbers
to a noncommutative associative algebra over the real numbers (a skew-field)
with elements
$$ q
=q_1+ q_2i +q_3 j +q_4 k 
=(q_1+ q_2i) +(q_3  +q_4 i)j
\in \HH, \qquad q_j\in\RR. $$
and (see \cite{R14} for the basic theory of $\HH$, which is assumed)
$$ i^2=j^2=k^2-1, \quad
ij=k, \quad jk=i,\quad ki=j, \quad ji=-k,\quad kj=-i,\quad ik=-j. $$
The Lagrange interpolant $Lf$ to a function $f$ at $n+1$ distinct points 
$x_0,x_1,\ldots,x_n$ 
in $\RR$ or $\CC$ is the unique polynomial of degree $n$ matching
the values of $f$ at these points, which can be given explicitly as
\begin{equation}
\label{Lagrangeformulas}
Lf (x) := \sum_j \ell_j(x) f(x_j) , \qquad
\ell_j(x):=\prod_{j\ne k} {x-x_k\over x_j-x_k}.
\end{equation}
Formally, the above formula makes sense for points in $\HH$, giving an interpolant.
However, due to the noncommutativity of the quaternions, the Lagrange polynomials
depend on the order in which the product is evaluated. This is the first indication
that the quaternionic polynomials of degree $n$ (as a right $\HH$-module) might 
have a dimension greater than $n+1$ (depending on how they are defined). 
To resolve this impasse one could
\begin{itemize}
\item Look for an interpolant from a fixed $(n+1)$-dimensional subspace
of quaternionic polynomials of degree $n$.
\item Seek a ``best'' choice of Lagrange polynomials $\ell_j$ for given interpolation points,
which would implicitly define an $(n+1)$-dimensional subspace of interpolants
that is related to the geometry of the points.
\end{itemize}

The first approach has been considered by \cite{B15}, which we discuss in 
the next section. Our approach is the second. The essential features of each are
(respectively):

\begin{itemize}
\item Interpolation is not possible for all configurations of points. The condition
for unique interpolation and the interpolation space are not translation invariant.
The Lagrange polynomials $\ell_j$ may have zeros which are not interpolation points. 
It is possible to develop a Newton form and notion of divided difference.
\item Interpolation is possible for all configurations of points, and the interpolation 
space depends continuously on the points. The interpolant is translation invariant. 
The Lagrange polynomials $\ell_j$ can be chosen to be zero only 
at the interpolation points, and a  Newton form can be developed.
\end{itemize}

We want 
polynomials and functions of a quaternionic variable 
to be an
``$\HH$-vector space''. To do this, we view such spaces a right $\HH$-module
(and co-opt the language of linear algebra). Linear maps
(which are technically $\HH$-homomorphisms) then act on the left,
with the usual algebra of matrices then extending in the obvious way
(cf.\ \cite{R14}). 
The Lagrange interpolant,
as defined above,
is an $\HH$-linear map, since
$$ L(f\ga+g\gb)(x) 
= \sum_j \ell_j(x)(f\ga+g\gb)(x_j)
= \sum_j \ell_j(x)(f(x_j)\ga+g(x_j)\gb)= Lf(x)\ga+Lg(x)\gb. $$ 

\section{Lagrange interpolation from $\HH[z]$}


Bolotnikov \cite{B15}, \cite{B20} uses the  
formal polynomials
$\HH[z]$ in $z$  
$$ f(z)=z^n f_n+\cdots+z f_1+f_0, \qquad f_0,\ldots,f_n\in\HH, $$
on which a left and right evaluation at $a\in\HH$ are defined by
$$ f^{e_l}(a):=\sum_j a^j f_j, \qquad 
f^{e_r}(a):=\sum_j f_ja^j. $$
For $\HH[z]$ as right vector space, left evaluation is linear
but right evaluation is not.

The (left) Lagrange interpolation of \cite{B15} is from the $(n+1)$-dimensional subspace
of polynomials of degree $n$ in $\HH[z]$ to left evaluation 
at $n+1$ points in $\HH$.
Let us consider an example, to see the nature of this interpolation.

\begin{example} There is a unique linear interpolant $p(z)=p_0+zp_1$ to a function $f$ at 
any distinct points $a,b\in\HH$ given by the Lagrange polynomial formula
\begin{align*}
 p(z) &= (z-b)(a-b)^{-1} f(a) +(z-a)(b-a)^{-1}f(b) \cr
&= \left(b(b-a)^{-1} f(a) a(a-b)^{-1}f(b)\right)
+ z \left((a-b)^{-1} f(a) -a(b-a)^{-1}f(b)\right). 
\end{align*}
Now we consider interpolation at the three points $i,j,c\in\HH$. 
Up to a scalar, the quadratic 
Lagrange polynomial from $\HH[z]$ which is zero at $i$ and $j$ is
$$ p(z)=1+z^2, $$
while those given by the Lagrange polynomial formula are
$$ p_1(x)=(x-i)(x-j), \qquad p_2(x)=(x-j)(x-i). $$
We note that $p$ is zero at all quaternions $q$ with $q^2=-1$,
equivalently, $\Re(q)=0$, $|q|=1$, e.g., $ q=k$, 
whereas $p_1$ and $p_2$ are zero precisely at $q=i,j$,
and they are not the same polynomial since $p_1(1)\ne p_2(1)$.
The theory of \cite{B15} is based on the Euclidean algorithm 
for the associative multiplication on $\HH[z]$ given by 
$$ \Bigl(\sum_j z^j a_j \Bigr)*\Bigl(\sum_k z^k b_k\Bigr) 
:= \sum_{j,k}  z^{j+k} a_jb_k. $$
As an example, the ``root'' $z=i$ of $p(z)$ gives a ``linear factor'' as follows
$$ (z^2+1)-\left((z-i)*z+(z-i)*i\right) =0 \Implies
z^2+1=(z-i)*(z+i). $$
Note that $f(z):=(z-i)*(z-j)=z^2-z(i+j)+k$ has $z=i$ as a left root,
i.e., $f^{e_l}(i)=0$, but not $z=j$ (which is a right root).
 \end{example}

Two quaternions $q_1$ and $q_2$ are said to be {\bf similar} (the term
{\it equivalent} is used in \cite{B15}) if $q_2=aq_1a^{-1}$ for some nonzero $a\in \HH$.
This is equivalent to $\Re(q_1)=\Re(q_2)$ and $|q_1|=|q_2|$. Hence
 $i,j,k$ are similar and $1$ and either of $j,k$ are not. 

The left point evaluations $\gd_a:\HH[z]\to\HH:f\mapsto f^{e_l}(a)$, $a\in\HH$, 
are $\HH$-linear functionals on the right vector space $\HH[z]$. These span 
a left vector space, with linear dependencies given by following:

\begin{lemma} (\cite{B15} Lemma 3.1) 
For $f\in\HH[z]$ and $a,b,c\in H$ distinct and similar
\begin{equation}
\label{leftevaldeps}
f^{e_l}(c) = (c-b)(a-b)^{-1} f^{e_l}(a) +(c-a)(b-a)^{-1}f^{e_l}(b),
\end{equation}
i.e.,
\begin{equation}
\label{leftevalfunctdeps}
\gd_c = (c-b)(a-b)^{-1} \gd_a +(c-a)(b-a)^{-1} \gd_b.
\end{equation}
\end{lemma}

To see this in play, we consider left quadratic Lagrange interpolation from $\HH[z]$.

\begin{example}
We seek a quadratic polynomial $p(z)=p_0+zp_1+z^2p_2$ which left Lagrange interpolates
a function $f$ at distinct points $a,b,c\in \HH$, i.e.,
\begin{align*}
p^{e_l}(a) = p_0+ap_1+a^2 p_2 &= f(a), \cr
p^{e_l}(b) = p_0+bp_1+b^2 p_2 &= f(b), \cr
p^{e_l}(c) = p_0+cp_1+c^2 p_2 &= f(c). 
\end{align*}
%
Gauss elimination gives the row echelon form
\begin{align*}
 p_0+ap_1+a^2 p_2 &= f(a), \cr
 p_1+(b-a)^{-1}(b^2-a^2) p_2 &= (b-a)^{-1}(f(b)-f(a)), \cr
 \{(c-a)^{-1}(c^2-a^2)-(b-a)^{-1}(b^2-a^2)\} p_2 &= (c-a)^{-1}(f(c)-f(a))-(b-a)^{-1}(f(b)-f(a)), 
\end{align*}
and so there is a unique interpolant to every $f$ if and only if
$$ (c-a)^{-1}(c^2-a^2)-(b-a)^{-1}(b^2-a^2) \ne 0. $$
It is easily seen that equality above is equivalent
to taking $f(z)=z^2$ in (\ref{leftevaldeps}), i.e.,
$$ c^2=(c-b)(a-b)^{-1}a^2+(c-a)(b-a)^{-1} b^2 \Iff
(c-a)^{-1}(c^2-a^2)=(b-a)^{-1}(b^2-a^2). $$
\end{example}

The Gauss elimination argument above shows that a necessary condition 
for left (or right) Lagrange interpolation by a polynomial in $\HH[z]$ of 
degree $n$ to any $f$ to $n+1$ distinct points in $\HH$ is that
no three of the points are similar.
This is in fact sufficient.

\begin{theorem} (\cite{B15} Theorem 3.3)
 Left (or right) Lagrange interpolation from the polynomials of degree $n$
in $\HH[z]$ to $n+1$ points in $\HH$ is uniquely possible if and only if the no 
three of the points are similar, i.e., have the same modulus and real part.
\end{theorem}


\begin{corollary} For a set $A\subset\HH$ the linear functionals
$\{\gd_a\}_{a\in A}$ given by 
$$\gd_a:\HH[z]\to\HH:f\mapsto f^{e_l}(a), $$
are $\HH$-linearly independent if and only if no three of them are similar.
\end{corollary}

\begin{proof} If three of the points $a,b,c\in A$ are similar, 
then we have the nontrivial linear dependency (\ref{leftevalfunctdeps}).
Conversely, suppose that no three points are similar.
Take a linear combination
$$ \sum_{j=0}^n c_j\gd_{a_j}=0, \qquad a_j\in A, \quad c_j\in \HH, $$
and apply both sides of this to the unique Lagrange interpolant to the function 
which is zero at all the points in $\{a_0,a_1,\ldots,a_n\}$, except $a_j$, 
to conclude that $c_j=0$.
\end{proof}

These results were developed from the notion of {\it $P$-independence} \cite{L86}, \cite{Ll88},
i.e.,
the set of $n+1$ points $A$ is (left) $P$-independent (polynomial independent) if 
the linear functionals $\{\gd_a\}_{a\in A}$ above are linearly independent, 
or, equivalently, there is a subspace of $\HH[z]$ of dimension $n+1$ from which unique
(left) Lagrange interpolation is possible. This has recently been explored in
the multivariate setting \cite{MK19}, \cite{M20}.

\begin{corollary} There is a unique quadratic left Lagrange interpolant
from $\HH[z]$ to the distinct points $a,b,c\in\HH$ if and only if the points
are not all similar. 
The condition for the points to be similar can be expressed as
\begin{equation}
\label{nonsymmthreeptcdn}
(c-a)^{-1}(c^2-a^2)=(b-a)^{-1}(b^2-a^2).
\end{equation}
\end{corollary}

A symmetric form of (\ref{nonsymmthreeptcdn})
can be obtained by evaluating the symmetrised form
of the linear dependence (\ref{leftevalfunctdeps})
at $f(z)=z^2$.




The condition for unique left Lagrange interpolation is not translation invariant,
except for some real translations or interpolation to less than three points.
For example, quadratic interpolation at $i,j,k$ is not possible, but it is possible at 
$2i,j+i,k+i$ (since $|2i|=2\ne\sqrt{2}=|j+i|$). In a similar vein, the polynomials of 
degree $n$ in $\HH[z]$ when viewed as functions $f:\HH\to\HH:q\mapsto f_0+qf_1+\cdots+q^nf_n$
are not translation invariant, e.g.,
$$ (q+a)^2= q^2+qa+aq+a^2 = q^2+qb+a^2, \qquad\forall q\in\HH, $$
for some $b\in\HH$, if and only if $a$ is real.



Since there is a unique one-dimensional subspace of polynomials of degree $n$ in $\HH[z]$ 
whose left evaluation at $n$ points (with no three similar) is zero, a Newton form for (left) Lagrange 
interpolation can be developed. 



\section{Quaternionic polynomials}

It is now time to understand the nature of the quaternionic polynomials, 
viewed as functions $\HH\to\HH$, obtained from the formula 
(\ref{Lagrangeformulas}) for the Lagrange polynomials.
These involve polynomials of the form
$$ q\mapsto \ga_0 q \ga_1 q \ga_2 \cdots q \ga_{r-1} q \ga_r, \qquad
\ga_j\in\HH, $$
which \cite{S79} calls a {\bf quaternionic monomial} of degree $r$. 
We define the $\HH$-span of these monomials to be $\Hom_r(\HH)$ the 
{\bf homogeneous polynomials} of degree $r$, and $\Pol_n(\HH)$ the 
{\bf polynomials} of degree $k$ to be the $\HH$-span of the 
homogeneous polynomials of degrees $\le n$. 
These definitions extend to multivariate polynomials $\Hd\to\HH$, where each 
occurrence of $q$ in the formula for a monomial is replaced by some coordinate $q_j$.

It is clear from the definitions, that the quaternionic polynomials are a graded
ring, i.e., the product of homogeneous polynomials of degrees $j$ and $k$ is
a homogeneous polynomial of degree $j+k$. To understand the dimensions of these
spaces, we write $q\in\HH$ as
$$ q=t+ix+jy+kz, \qquad t,x,y,z\in\RR, $$
and observe (see \cite{S79}) that
\begin{align}
\label{txyzeqns}
t &= {1\over4} ( q - iqi - jqj - kqk ),  \cr
x &= {1\over4i}( q - iqi + jqj + kqk ), \cr
y &= {1\over4j}( q + iqi - jqj + kqk ), \cr
z &= {1\over4k}( q + iqi + jqj - kqk ).
\end{align}
Hence $t,x,y,z$ are homogeneous monomials (in $q$), as are $\overline{q}$ and $|q|^2=q\overline{q}$, i.e.,
$$ \overline{q}=t-ix-jy-kz=-{1\over2}(q+iqi+jqj+kqk), $$
$$ |q|^2=q\overline{q}
=-{1\over2}(q^2+(qi)^2+(qj)^2+(qk)^2). $$
Every monomial of degree $r$ can be written as a homogeneous polynomial of
degree $r$ in the (real) variables $t,x,y,z$ with quaternionic coefficients.
The monomials in $t,x,y,z$ are linearly independent over $\HH$ by the usual
argument (of taking Taylor coefficients), and so we have
\begin{equation}
\label{Homrdim}
\dim_\HH(\Hom_r(\HH))=\dim_\RR(\Hom_r(\RR^4))= {r+3\choose 3}, 
\end{equation}
\begin{equation}
\label{Polnrdim}
\dim_\HH(\Pol_n(\HH))=\dim_\RR(\Pol_k(\RR^4))= {n+4\choose 4}.
\end{equation}
In particular, the vector spaces of constant, linear and quadratic polynomials
of a single quaternionic variable have dimensions $1,5$ and $15$, respectively.
This contrasts sharply with the real and complex cases, where the 
dimensions are $1,2$ and $3$.

\begin{example}
In view of (\ref{txyzeqns}),
a basis for the linear polynomials is given by $1,q,iq,jq,kq$. 
There is a unique linear Lagrange interpolant to any five points
$x_0,\ldots,x_5$, which are affinely independent as points in $\RR^4$.
An explicit formula for the Lagrange interpolant
$p(q)=p_0+qp_1+iqp_2+jqp_3+kqp_4$ can be obtained by solving the
``linear system''
$$ p(x_j)=p_0+x_jp_1+ix_jp_2+x_jp_3+kx_jp_4 =f(x_j), \qquad 0\le j\le 4, $$
for $p_0,\ldots,p_4\in\HH$ in the skew-field $\HH$. 
The corresponding Lagrange polynomials are 
the barycentric coordinates for the interpolation points. If the points 
are taken to be $0,1,i,j,k$, then the interpolant can be written as
$$ (1-t-x-y-z)f(0)+tf(1)+xf(i)+yf(j)+zf(k). $$
From this a (multivariate) Bernstein interpolant could be developed, 
if desired.
\end{example}

As is evident from this example, Lagrange interpolation from $\Pol_n(\HH)$
is essentially interpolation from $\Pol_n(\RR^4)$, with
the additional feature that formulas can be developed using a single
quaternionic variable $q$, or two complex variables $v,w$, where
\begin{equation}
\label{Cayley–Dickson}
q=v+jw, \qquad v=t+ix={1\over2}(q-iqi), 
\qquad y-iz={1\over2}(-jq+kqi). 
\end{equation}

We now consider interpolation from subspaces of $\Pol_n(\HH)$ that are
a quaternionic analogue of the holomorphic functions. 
A function $f:\HH\to\HH$ is said to be {\bf regular} if
it is in the kernel of the {\bf Cauchy-Feuter operator} $2\partial_\ell$, 
i.e.,
$$ 2\partial_\ell f:= {\partial f\over \partial t}
+i{\partial f\over \partial x}
+j{\partial f\over \partial y}
+k{\partial f\over \partial z}= 0, $$
and to be {\bf harmonic} if
$$ \gD f:= {\partial^2 f\over \partial t^2}
+{\partial^2 f\over \partial x^2}
+{\partial^2 f\over \partial y^2}
+{\partial^2 f\over \partial z^2}= 0. $$
If $f$ is regular, then it is harmonic.
The dimensions of $\Reg_n(\HH)$ and $\Harm_n(\HH)$,
the regular and harmonic homogeneous polynomials of degree $n$, 
are
$$ 
\dim_\HH(\Reg_n(\HH))={1\over 2}(n+1)(n+2), \qquad
\dim_\HH(\Harm_n(\HH)) = (n+1)^2. $$

\begin{example} 
\label{regularlinear}
The harmonic polynomial $q=t+ix+jy+kz$ is not regular, since
$$ {\partial q\over \partial t}
+i{\partial q\over \partial x}
+j{\partial q\over \partial y}
+k{\partial f\over \partial z}
= 1 +i(i)+j(j)+k(k) 
=-2\ne 0. $$
A basis for (the right $\HH$-vector space) $\Reg_1(\HH)$ 
is given by $t+ix,t+jy,t+kz$.
\end{example}

An explicit basis 
$\{P_{k\ell}^n-jP_{k-1,\ell}^n\}_{0\le k\le\ell\le1}$,
for 
 $\Reg_n(\HH)$ is given in \cite{S79},
where 
$$ P_{k\ell}^n(v+jw):=\sum_r (-1)^r v^{[n-k-\ell+r]}\overline{v}^{[r]}w^{[k-r]}\overline{w}^{[\ell-r]}, \quad
v,w\in\CC, \qquad
z^{[j]}:=\begin{cases}
{z^j\over j!}, & j\ge0; \cr
0, & j<0.
\end{cases}
$$
Here 
$$ q=t+ix+jy=kz=(t+ix)+j(y-iz)=v+jw, $$
so for $n=1$, 
$P_{00}^1(q)=v$, $P_{01}^1(q)=\overline{w}$, $P_{11}^1(q)=-\overline{v}$.
But $Q_{01}^1(q)=\overline{w}$ is not regular. Interchanging
$k$ and $\ell$ in the formula (presumably a typo), gives $P_{01}^1(q)=w$,
and the basis
$$ v=t+ix, \qquad w=y-iz=(t+jy)(-j)+(t+kz)j, $$
$$ -\overline{v}-j{w} =-(t-ix)-j(y-iz) 
= (t+ix)-(t+jy)-(t+kz).  $$

The product of regular functions is not (in general) regular, 
since
$$ 2\partial_\ell \bigl( (t+ix)(t+jy)\bigr) = 2ix, \qquad
2\partial_\ell \bigl( (t+jy)(t+ix)\bigr) = 2jy. $$
However, multiplying the basis 
of Example \ref{regularlinear}
by  $i,j,k$, 
to get $it-x,jt-y,kt-z$, we have
$$ 2\partial_\ell \bigl( (it-x)(jt-y)\bigr) = 2kt, \qquad
2\partial_\ell \bigl( (jt-y)(it-x)\bigr) = -2kt, $$
so that the average
$ {1\over2} \{(it-x)(jt-y)+(jt-y)(it-x)\} $
is regular. In this way, \cite{S79} gives a basis for $\Reg_n(\HH)$
consisting of symmetrised products of the linear regular polynomials
$it-x,jt-y,kt-z$, where the factors occur $n_1,n_2,n_3$ times, 
$n_1+n_2+n_3=n$. 

Though the (Feuter) regular polynomials are a proper subspace of 
the quaternionic polynomials (which share some aspects of
holomorphic polynomials, but not that power series, since $q$ is 
not regular), 
we do not readily see a practical 
way to interpolate from them. 
There has recently been considerable interest in 
Cullen-regular functions \cite{GS07}, which do have power series expansions
(around 0 or a real centre), and so correspond to the $(n+1)$-dimensional 
subspace of formal polynomials $\HH[z]$ in $\Pol_n(\HH)$, 
which we have already discussed.

\section{Multivariate Lagrange interpolation}

We now consider interpolation methods derived from 
(\ref{Lagrangeformulas}). As already discussed, 
this is essentially 
multivariate polynomial interpolation to functions $\RR^4\to\HH$.

\begin{example}
The order in which the factors of $\ell_j(x)$ are calculated is important.
For $x_0=i$, $x_1=j$, $x_2=k$, we might take ``$\ell_0(x)$'' to be
$$ p(x) := (x-j)(x-k)(i-j)^{-1}(i-k)^{-1}, \qquad
p(i) = {1\over2}(-1-i-j-k). $$
This is not $1$ at $x_0$, and so care must be taken with the order
of multiplication in (\ref{Lagrangeformulas}). 
Natural choices for the multiplication order are to evaluate each quotient first (with right scalar 
multiplication), or to take the product of the numerators, and then right multiply this
by the inverse of its value at $x_j$, in concrete terms for $\ell_0$ for three points
either of
$$ (x-x_1)(x_0-x_1)^{-1} (x-x_2)(x_0-x_2)^{-1}, \qquad
(x-x_1) (x-x_2) \bigl( (x-x_1) (x-x_2))\bigr)^{-1}. $$
\end{example}

Either of the choices for computing $\ell_j(x)$ suggested above give
\begin{itemize}
\item A unique linear
Lagrange interpolation operator to any points $x_0,\ldots,x_n\in\HH$ from the 
$(n+1)$-dimensional subspace $\spam\{\ell_j\}$ of $\Pol_n(\HH)$,
where the Lagrange polynomials have precisely $n$ zeros. 
\item The operator depends continuously on the interpolation points.
\item
It could be ``symmetrised'' to obtain an operator which doesn't
depend on the ordering of the points (though the Lagrange polynomials
might now have additional zeros).
\end{itemize}
 In this vein, we now define a generic 
polynomial of degree $n$ which is zero at $n$ points.
For points $x_1,\ldots,x_n\in\HH$, let 
$$ p_{\{x_1,\ldots,x_n\}} (x) : =
{1\over n!} \sum_{\gs\in S_{n}} (x-x_{\gs1})(x-x_{\gs2})\cdots
(x-x_{\gs n}), $$
where $S_n$ is the symmetric group. This polynomial of degree $n$ is zero at the 
points $x_1,\ldots,x_n$, and (by construction) does not depend on their ordering.

We now present two possible choices for the Lagrange polynomials:
\begin{equation}
\label{ljchoiceI}
\ell_j(x):= p(x) p(x_j)^{-1}, \qquad p(x)=p_{\{x_0,\ldots,x_n\}\setminus\{x_j\}}(x), 
\end{equation}
provided that $p(x_j)\ne0$, and
\begin{equation}
\label{ljchoiceI}
 \ell_j(x):= {1\over n!}\sum_{\gs\in S_{n+1}\atop \gs j=j}
\prod_{k\ne j} (x-x_{\gs k})(x_j-x_{\gs k})^{-1},
\end{equation}
where the factors above are multiplied in the order $k=0,1,\ldots,n$ (or any fixed order).
Both are independent of the point ordering, and depend continuously on the points.
Let $L_\gTh$ be the corresponding Lagrange interpolation operator
$$ L_\gTh f (x):=\sum_j \ell_j(x) f(x_j), $$
for the points $\gTh=\{x_0,\ldots,x_n\}\subset\HH$, which does not 
depend on their ordering. We have
\begin{itemize}
\item The interpolation operator $L_\gTh$ depends continuously on the points $\gTh$.
\item The interpolation space $\Pi_\gTh:=\ran(L_\gTh)$ depends continuously 
on the points $\gTh$.
\item The interpolation operator is translation invariant, i.e., 
$$ L_{\gTh+a} f(x) = L_\gTh\bigl(f(\cdot+a)\bigr)(x-a). $$
\end{itemize}

We compare this Lagrange interpolation with the two most prominent multivariate generalisations
of univariate Lagrange interpolation, where the interpolation points are not in some 
predetermined geometric configuration. 

Kergin interpolation \cite{MM80}
interpolates function values at $n+1$ points in $\Rd$ by a polynomial of degree $n$,
with other ``mean-value'' interpolation conditions (see \cite{W97}) that depend
continuously on the points also matched. Here the interpolation space $\Pol_n(\Rd)$ is fixed,
and hence depends continuously on the points. Kergin interpolation has also been
extended to $\Cd$ (see \cite{F97}). Using the identifications $\HH\approx\RR^4$, $\HH\approx \CC^2$
one can defined a Kergin interpolation to functions $\HH\to\HH$ from the whole space $\Pol_n(\HH)$
(with additional interpolation conditions). The explicit formulas for Kergin interpolation involve
derivatives of $f$ (the operator is defined for $C^n$-functions), and are not as easily 
computed as ours.

Least interpolation \cite{BR92} is a very general method, which seeks an interpolation space
$\Pi_\gTh$ of dimension $n+1$ to the $n+1$ points $\gTh$, which has polynomials of lowest (least) degree. 
It has many nice properties that include the continuity properties listed above, but 
there is no explicit formula. It could be
applied to functions $\HH\to\HH$ in the same way that Kergin interpolation can be. It might
even be possible to develop a least interpolation for ``polynomials in $\HH^d$'', once such 
a theory is developed. The least interpolation has the advantage that polynomials of low
degree are used. In particular, the Lagrange polynomials are a partition of unity, i.e.,
$$ \sum_j \ell_j =1. $$
For the Lagrange polynomials that we have proposed, this may not be the case.
Indeed, the set of all possible ``nice'' Lagrange polynomials $\ell_j$ for $x_j$, i.e, those polynomials 
$p=p_{\{x_0,\ldots,x_n\}}$ of degree $n$,
with $p(x_k)=\gd_{jk}$ and $p$ not depending on the order of the points, form
an affine subspace of $\Pol_n(\HH)$, from which we have suggested two choices.
It may be that requiring, in addition, the partition of unity property, gives a unique
choice, but we have not pursued this. 

Since our Lagrange interpolation is effectively interpolation to functions $\RR^4\to\HH$,
we can define an interpolation operator to functions $\RR^4\to\RR$ in the natural way, i.e.,
$$ Lf:=\Re(L_\gTh f)= \sum_j \Re(\ell_j) f(x_j) , \qquad
\hbox{where $f:\HH\to\RR$}, $$
and $\gTh$ is the points viewed as a subset of $\HH$. 
Heuristically, the real Lagrange polynomials $\hat\ell_j =\Re(\ell_j)$ are 
more likely to be ``nice'', e.g., form a partition of unity, have 
no extra zeros, or coincide for both 
choices, since there is the ``commutativity relation''
$$ \Re(ab) =\Re(ba), \qquad a,b\in\HH. $$

In a similar way, one could define a Lagrange interpolation operator to functions
$\CC^2\to\CC$, by using the Cayley-Dickson construction (\ref{Cayley–Dickson}).

It is also possible to develop Lagrange interpolants through a Newton form
(either for functions $\HH\to\HH$ or $\RR^4\to\RR$), and an associated theory
of divided differences. For example, if $L_{n-1} f$ is a Lagrange interpolant
at $x_1,\ldots,x_{n-1}\in \HH$, then
$$ L_n f (x) := L_{n-1} f(x) + p_n(x) [x_0,\ldots,x_n]f,  \qquad
\hbox{where $p\in\Pol_n(\HH)$, \quad 
$p(x_j)=\gd_{jn}$}, $$
gives a Lagrange interpolant $L_nf$ to $f$ at the points $x_0,\ldots,x_n$, 
and a ``divided difference'' $[x_0,\ldots,x_n]f\in\HH$. 
Choices for $p$ could include $p_{\{x_0,\ldots,x_{n-1}\}}$ or $\ell_n$ 
(for $x_0,\ldots,x_n$).
The map $f\mapsto [x_0,\ldots,x_n]f\in\HH$ is an $\HH$-linear functional. 
We have not invesigated its divided difference type properties any further.


\section{Concluding remarks}

This work came about when investigating tight frames and spherical 
designs for $\Hd$ (see \cite{W18}, \cite{W20}). It soon became apparent
that the theory of quaternionic polynomials of one and several variables
is involved, and not widely known. There is considerable work in at least
two different directions. One is a formal (algebraic) approach dating back 
to \cite{O33}, \cite{L86}, where point evaluation and multiplication of 
polynomials are appropriately defined, and the other views the polynomials
as functions, with the usual point evaluation and (pointwise) multiplication.

Here we have shown that

\begin{itemize}
\item The space of quaternionic polynomials required
for the classical formula (\ref{Lagrangeformulas})
for Lagrange interpolation 
to make sense has a high dimension (\ref{Polnrdim}).
\item Lagrange interpolation methods with some geometric 
properties, e.g., translation invariance, are essentially
multivariate polynomial interpolation methods.
\item Many basic questions about quaternionic polynomials
of one (and several) variables remain, e.g., the existence
of Lagrange polynomials with no extra zeros which form 
a partition of unity.
\end{itemize}
Functions of quaternionic variables have long been used in physics, and geometric design
(cf.\ \cite{FGMS19}).  We hope this paper gives some insight into polynomials of one or more
quaternionic variables, and their use in interpolation and cubature (spherical designs).


\bibliographystyle{alpha}
\bibliography{references}

\begin{thebibliography}{MPnK19}

\bibitem[Bol15]{B15}
Vladimir Bolotnikov.
\newblock Polynomial interpolation over quaternions.
\newblock {\em J. Math. Anal. Appl.}, 421(1):567--590, 2015.

\bibitem[Bol20]{B20}
Vladimir Bolotnikov.
\newblock Lagrange interpolation over division rings.
\newblock {\em Comm. Algebra}, 48(9):4065--4084, 2020.

\bibitem[dBR92]{BR92}
Carl de~Boor and Amos Ron.
\newblock The least solution for the polynomial interpolation problem.
\newblock {\em Math. Z.}, 210(3):347--378, 1992.

\bibitem[FGMS19]{FGMS19}
Rida~T. Farouki, Graziano Gentili, Hwan~Pyo Moon, and Caterina Stoppato.
\newblock Minkowski products of unit quaternion sets.
\newblock {\em Adv. Comput. Math.}, 45(3):1607--1629, 2019.

\bibitem[Fil97]{F97}
Lars Filipsson.
\newblock Complex mean-value interpolation and approximation of holomorphic
  functions.
\newblock {\em J. Approx. Theory}, 91(2):244--278, 1997.

\bibitem[GS07]{GS07}
Graziano Gentili and Daniele~C. Struppa.
\newblock A new theory of regular functions of a quaternionic variable.
\newblock {\em Adv. Math.}, 216(1):279--301, 2007.

\bibitem[Lam86]{L86}
T.~Y. Lam.
\newblock A general theory of {V}andermonde matrices.
\newblock {\em Exposition. Math.}, 4(3):193--215, 1986.

\bibitem[LL88]{Ll88}
T.~Y. Lam and A.~Leroy.
\newblock Vandermonde and {W}ronskian matrices over division rings.
\newblock {\em J. Algebra}, 119(2):308--336, 1988.

\bibitem[{Mar}20]{M20}
Umberto {Mart{\'\i}nez-Pe{\~n}as}.
\newblock {Theory and applications of linearized multivariate skew
  polynomials}.
\newblock {\em arXiv e-prints}, page arXiv:2001.01273, January 2020.

\bibitem[MM80]{MM80}
Charles~A. Micchelli and Pierre Milman.
\newblock A formula for {K}ergin interpolation in {${\bf R}^{k}$}.
\newblock {\em J. Approx. Theory}, 29(4):294--296, 1980.

\bibitem[MPnK19]{MK19}
Umberto Mart\'{\i}nez-Pe\~{n}as and Frank~R. Kschischang.
\newblock Evaluation and interpolation over multivariate skew polynomial rings.
\newblock {\em J. Algebra}, 525:111--139, 2019.

\bibitem[Ore33]{O33}
Oystein Ore.
\newblock Theory of non-commutative polynomials.
\newblock {\em Ann. of Math. (2)}, 34(3):480--508, 1933.

\bibitem[Rod14]{R14}
Leiba Rodman.
\newblock {\em Topics in quaternion linear algebra}.
\newblock Princeton Series in Applied Mathematics. Princeton University Press,
  Princeton, NJ, 2014.

\bibitem[Sud79]{S79}
A.~Sudbery.
\newblock Quaternionic analysis.
\newblock {\em Math. Proc. Cambridge Philos. Soc.}, 85(2):199--224, 1979.

\bibitem[Wal97]{W97}
Shayne Waldron.
\newblock Integral error {formul\ae} for the scale of mean value interpolations
  which includes {K}ergin and {H}akopian interpolation.
\newblock {\em Numer. Math.}, 77(1):105--122, 1997.

\bibitem[Wal18]{W18}
Shayne F.~D. Waldron.
\newblock {\em An introduction to finite tight frames}.
\newblock Applied and Numerical Harmonic Analysis. Birkh\"{a}user/Springer, New
  York, 2018.

\bibitem[Wal20]{W20}
Shayne Waldron.
\newblock Tight frames over the quaternions and equiangular lines.
\newblock preprint, 1 2020.

\end{thebibliography}
\nocite{*}

\end{document}